\documentclass{article} 

\usepackage{comment}

\usepackage{extarrows}

\usepackage[usenames,dvipsnames]{pstricks}
\usepackage{epsfig}
\usepackage{pst-grad} 
\usepackage{pst-plot} 
\usepackage[space]{grffile} 
\usepackage{etoolbox} 
\makeatletter 
\patchcmd\Gread@eps{\@inputcheck#1 }{\@inputcheck"#1"\relax}{}{}
\makeatother

\usepackage{amsmath}
\usepackage{amssymb}
\usepackage{amsthm}

\usepackage{graphicx}

\usepackage{ifthen}
\usepackage{color}

\newcommand{\intav}[1]{\mathchoice {\mathop{\vrule width 6pt height 3 pt depth  -2.5pt
\kern -8pt \intop}\nolimits_{\kern -6pt#1}} {\mathop{\vrule width
5pt height 3  pt depth -2.6pt \kern -6pt \intop}\nolimits_{#1}}
{\mathop{\vrule width 5pt height 3 pt depth -2.6pt \kern -6pt
\intop}\nolimits_{#1}} {\mathop{\vrule width 5pt height 3 pt depth
-2.6pt \kern -6pt \intop}\nolimits_{#1}}}

\def\polhk#1{\setbox0=\hbox{#1}{\ooalign{\hidewidth\lower1.5ex\hbox{`}\hidewidth\crcr\unhbox0}}}

\newtheorem{theorem}{Theorem}

\newtheorem{definition}{Definition}
\newtheorem{lemma}{Lemma}
\newtheorem{corollary}{Corollary}
\newtheorem{proposition}{Proposition}
\newtheorem{remark}{Remark}
\newtheorem{assumption}{A}

\newcommand{\abs}[1]{\left\lvert#1\right\rvert}

\usepackage{enumerate}

\usepackage{setspace}
\begin{document}

\title{A Hessian-dependent functional with free boundaries and applications to mean-field games}

\author{Julio C. Correa and Edgard A. Pimentel}

\date{\today}

\maketitle

\begin{abstract}
\begin{spacing}{1.15}
\noindent We study a Hessian-dependent functional driven by a fully nonlinear operator. The associated Euler-Lagrange equation is a fully nonlinear mean-field game with free boundaries. Our findings include the existence of solutions to the mean-field game, together with H\"older continuity of the value function and improved integrability of the density. In addition, we prove the reduced free boundary is a set of finite perimeter. To conclude our analysis, we prove a $\Gamma$-convergence result for the functional.
\end{spacing}

\medskip


\medskip 

\end{abstract}

\vspace{.1in}

\section{Introduction}
We examine Hessian-dependent functionals of the form
\begin{equation}\label{eq_mainfunc}
	\mathcal{F}_{\Lambda,p}[u]:=\int_{B_1}F(D^2u)^pd x+\Lambda|\{u>0\}\cap B_1|,
\end{equation}
where $F:S(d)\to\mathbb{R}$ is a uniformly elliptic operator, $\Lambda>0$ is a fixed constant, $p>d/2$, and $S(d)\sim\mathbb{R}^\frac{d(d+1)}{2}$ stands for the space of symmetric matrices of order $d$. Our results include the existence of minimizers for \eqref{eq_mainfunc}, amounting to the existence of solutions to a fully nonlinear mean-field game with free boundaries. We prove H\"older-continuity of minimizers and improved integrability of the density. We also prove the free boundary has finite perimeter. Finally, we establish a result on the $\Gamma$-convergence of $\mathcal{F}_{\Lambda,p}$ and examine its consequences.

The functional in \eqref{eq_mainfunc} is inspired by the usual one-phase Bernoulli problem, driven by the Dirichlet energy. To a limited extent, we understand $\mathcal{F}_{\Lambda,p}$ as a Hessian-dependent counterpart of that problem. See \cite{AltCaf_1981}; see also \cite{CafSalbook}.

The analysis of \eqref{eq_mainfunc} relates closely with the system
\begin{equation}\label{eq_mfgmain}
	\begin{cases}
		F(D^2u)=m^\frac{1}{p-1}&\hspace{.2in}\mbox{in}\hspace{.2in}B_1\cap\{u>0\}\\
		\left(F_{i,j}(D^2u)m\right)_{x_ix_j}=0&\hspace{.2in}\mbox{in}\hspace{.2in}B_1\cap\{u>0\},
	\end{cases}
\end{equation}
where $F_{i,j}(M)$ denotes the derivative of $F$ with respect to the entry $m_{i,j}$ of $M$. Here, the unknown is a pair $(u,m)$ solving the problem in a sense we make precise further. 

The system in \eqref{eq_mfgmain} amounts to the Euler-Lagrange equation associated with \eqref{eq_mainfunc}. We notice that \eqref{eq_mfgmain} satisfies an \emph{adjoint structure}. Its double-divergence equation is the formal adjoint, in the $L^2$-sense, of the linearized fully nonlinear problem. Due to such a distinctive pattern, we refer to \eqref{eq_mfgmain} as a fully nonlinear mean-field game with free boundary. Interpreting the first equation in \eqref{eq_mfgmain} as a Hamilton-Jacobi equation, it becomes natural to ask about the underlying stochastic optimal control problem. We do not examine this matter in the present paper.

Fully nonlinear elliptic operators and equations in the double-divergence form have been studied by many authors. An attempt to put together a comprehensive list of references on those topics is unrealistic. For that reason, we mention solely the monographs \cite{bkrsbook,ccbook} and the references therein.

Our analysis sits at the intersection of Hessian-dependent functionals, free boundary problems, and mean-field games systems. Hence we proceed with some context on those classes of problems. Hessian-dependent functionals play an essencial role in various contexts. From a purely mathematical viewpoint, they are useful to produce examples of conformally invariant energies. In dimension $d=4$, this is the case of 
\[
	J[u]:=\int_{B_1}\left(\Delta u\right)^2\mathrm{d}x,
\]
whose Euler-Lagrange equation is the biharmonic operator; see \cite{achang1,achang2}. 

When it comes to applications, we mention the realm of mechanics of solids. In particular, the analysis of energy-driven pattern formation and nonlinear elasticity. For example, Hessian-dependent models play a role in studying the occurrence of wrinkles in a twisted ribbon \cite{kohn2}. The energy modeling the system depends on the thickness of the ribbon, denoted with $h$, and two symmetric tensors $M$ and $B$. It has the form
\[
	\int_{B_1}|M(u,v)|^2+h^2|B(u,v)|^2\mathrm{d}x.
\]
Although $M$ depends on its arguments only through lower-order terms, the tensor $B$ depends on $\|D^2u\|$. Another instance where Hessian-dependent functionals appear is the analysis of blister patterns in thin films on compliant substrates \cite{kohn3}. Here the phenomena are modeled in terms of lower-order quantities, a small Hessian-dependent perturbation, and a parameter $h>0$. An important question concerning this class of problems is the limiting behavior $h\to 0$; in fact, one expects that the lower and upper bounds of the functional scale similarly. We refer the reader to \cite{maggi1,maggi2,venkataramani}. In this context, \eqref{eq_mainfunc} amounts to an energy penalized by the measure of the positive phase. See also \cite{kohnicm,avilesgiga1,avilesgiga2}.

As noticed before, one can state the Euler-Lagrange equation associated with \eqref{eq_mainfunc} in terms of the fully nonlinear mean-field game system with free boundaries \eqref{eq_mfgmain}. Mean-field games comprise a set of methods and techniques to model strategic interactions involving many players \cite{ll1,ll2,ll3}; see also \cite{LCDF}. At the intersection of partial differential equations (PDE), stochastic analysis and numerical methods, this class of problems has attracted the attention of several authors, who have developed the theory in various directions. 

Under assumptions on the stochastic dynamics governing the problem (e.g., independence of the Brownian motions among the population of players), it is possible to write a mean-field game in terms of a coupling. It comprises a Hamilton-Jacobi equation, accounting for the value of the game, and a Fokker-Planck equation describing the evolution of the population.  For further references on the topics, we mention the monographs \cite{cardaliaguet,bensoussan,pim1,Achdoubook}. Fully nonlinear mean-field games are the subject of \cite{AndPim_2020,jakobsen_2021}.

The interesting aspect in \eqref{eq_mfgmain} concerns the appearance of a free boundary. At least heuristically, the game is played only in the regions where the value function is strictly positive. Combined with the free boundary condition, \eqref{eq_mfgmain} models a game in which players optimize in the region where the value function is positive and might face extinction according to a flux condition endogenously determined.

Our first contribution is to prove the existence of solutions for the mean-field game system in \eqref{eq_mfgmain}. To that end, we impose natural assumptions on the data of the problem. Namely, we require the operator $F$ driving the problem to be uniformly elliptic (Assumption A\ref{A1}) and to satisfy a norm-like growth condition (Assumption A\ref{A3}). We also require the operator $F$ to be convex (Assumption (A\ref{A2}) and the boundary data $g$ to be a function in $W^{2,p}(B_1)$ (Assumption A\ref{A4}). We report our findings in the following theorem.

\begin{theorem}[Existence and regularity of solutions]\label{teo_existence}
Suppose Assumptions A\ref{A1}, A\ref{A2}, A\ref{A3}, and A\ref{A4}, to be detailed further, are in force. In addition, suppose $p<d<2p$.  Then there exists a solution $(u,m)\in W^{2,p}(B_1)\times L^1(B_1)$ to \eqref{eq_mfgmain}. Also, for every $\alpha\in(0,p/d)$, we have $u\in C^{\alpha}_{\rm loc}(B_1)$ and there exists $C_\alpha>0$ such that 
\[
	\|u\|_{C^\alpha(B_{1/2})}\leq C_\alpha\|g\|_{W^{2,p}(B_1)}.
\]
The constant $C_\alpha>0$ depends on the exponent $\alpha$.
\end{theorem}

If, in addition, $F$ is strictly convex and $p>2$, we can prove that $m$ is not only integrable but is indeed an $L^\frac{p}{p-1}$-function, with estimates; c.f. \cite{fabesstroock}. To establish Theorem \ref{teo_existence}, we start with the direct method in the calculus of variations and the existence of minimizers for \eqref{eq_mainfunc}. Then we turn our attention to \eqref{eq_mfgmain}. First, we resort to the theory of weak solutions available for equations in the double-divergence form. Finally, elements in the $L^p$-viscosity theory lead to the existence of solutions to the system. To complete the proof, we resort to a delicate application of Sobolev inequalities.

Once we have established the existence of solutions for \eqref{eq_mfgmain} and produced a regularity result, we examine the free boundary. Regularity results for the solutions build upon ingredients of geometric measure theory to ensure the reduced free boundary is a set of finite perimeter. We summarize our findings in this direction in the following result.

\begin{theorem}[Free boundary condition and finite perimeter]\label{teo_fb}
Let $u\in W^{2,p}_{\rm loc}(B_1)\cap W^{1,p}_g(B_1)$ be a local minimizer for \eqref{eq_mainfunc}, for $p>d/2$. Suppose Assumptions A\ref{A1}, A\ref{A2}, A\ref{A3} and A\ref{A4}, to be detailed further, are in force. Then the reduced free boundary, denoted with $\partial^*\{u>0\}$, is a set of finite perimeter. 
\end{theorem}
The remainder of this paper is organized as follows. Section \ref{subsec_assump} details our main assumptions, whereas Section \ref{subsec_prelim} gathers preliminary material and results. Section \ref{Sec:Existence} presents the proof of Theorem \ref{teo_existence}. In Section \ref{section_theorem2}, we examine the free boundary and put forward the proof of Theorem \ref{teo_fb}. A final section closes the paper with a $\Gamma$-convergence result and some consequences.

\section{Preliminaries}\label{Section:MainAssumptions}
This section presents the main assumptions under which we work and collects some preliminary notions and results.

\subsection{Main assumptions}\label{subsec_assump}
We proceed with a condition on the uniform ellipticity of the operator $F$.

\begin{assumption}[Uniform ellipticity]\label{A1}
We suppose the operator $F: S(d) \rightarrow$ $\mathbb{R}$ is $\lambda$-elliptic for some $\lambda\geq 1$. That is,
\begin{equation}\label{eq_lambdaeliptic}
	\frac1\lambda\|N\| \leq F(M+N)-F(M) \leq \lambda\|N\|
\end{equation}
for every $M, N \in S(d)$, with $N \geq 0 $. We also suppose $F(0)=0$.  Finally, we require $F_{ij}(M)=F_{ji}(M)$, for every $i,j=1,\ldots,d$, where $F_{ij}(M)$ stands for the derivative of $F$ with respect to the entry $(i,j)$ of $M$.
\end{assumption}

\begin{remark}\label{Remark:Assumptions}
Note that A\ref{A1}  implies a coercivity condition on $F$ over non-negative matrices. By taking $M \equiv 0$, the inequalities in \eqref{eq_lambdaeliptic} yield
\[
	\frac1\lambda\|N\| \leq F(N) \leq \lambda\|N\|
\]
for every $N \geq 0 $. 
\end{remark}

Next, we impose a convexity condition on the operator $F$.

\begin{assumption}[Convexity of the operator $F$]\label{A2}
We suppose the operator $F=F(M)$ to be convex with respect to $M$.
\end{assumption}

Part of our arguments requires $F$ to satisfy a coercivity condition in the entire $S(d)$. To that end, we strength A\ref{A1} as follows.
\begin{assumption}[Growth condition]\label{A3}
 We suppose there exists $\lambda\geq 1$ such that the operator $F$ satisfies
\[
	\frac1\lambda\|M\| \leq F(M) \leq \lambda\|M\|
\]
for every $M \in S(d)$.
\end{assumption}

The typical example of an operator $F=F(M)$, satisfying the former assumptions, depends on $M$ through its norm. For instance, let $A\in \mathbb{R}$ be a constant and consider
\[
	F(M):=A\|M\|.
\]
For a more general operator, including explicit dependence on the space-variable $x\in B_1$, we consider $A\in W^{2,p}_{\rm loc}(B_1)$ and define $F=F(x,M)$ as
\[
	F(x,M):=A(x)\|M\|.
\]
For further examples related to the previous one see \cite{kohnicm,kohn3}.

Notice that A\ref{A3} implies $F\geq 0$; this fact is important when studying the weak lower semicontinuity of our functional. We conclude this section with an assumption on the boundary data.

\begin{assumption}[Boundary data]\label{A4}
 We suppose the function $g\in W^{2,p}(B_1)$ is non-negative and non-trivial.
\end{assumption}

\subsection{Preliminary notions and results}\label{subsec_prelim} For the sake of completeness, we recall definitions and former results we use throughout the manuscript. We continue with a definition 
\begin{definition}[Affine Sobolev spaces]
 Let $p>d/2$ and $g\in W^{2,p}_{\rm loc}(B_1)$. We say that 
 \[
 	u\in W^{2,p}_{\rm loc}(B_1)\cap W^{1,p}_g(B_1)
 \]
 if $u\in W^{2,p}_{\rm loc}(B_1)$ and $u-g\in W^{1,p}_0(B_1)$.
\end{definition}
From a PDE perspective, having $u\in W^{2,p}_{\rm loc}(B_1)\cap W^{1,p}_g(B_1)$ is tantamount to prescribe $u=g$ on $\partial B_1$ in the Sobolev sense. This interpretation will be helpful when relating \eqref{eq_mainfunc} and \eqref{eq_mfgmain}.

As usual in the literature on mean-field games \cite{ll1,ll2,ll3}, a solution to \eqref{eq_mfgmain} relies on two distinct definitions -- namely, the notions of viscosity and weak (distributional) solutions. We proceed by recalling the definition of $L^p$-viscosity solution of a fully nonlinear elliptic equation;  see \cite[Definition 2.1]{CafCraKocSwi_1996}.

\begin{definition}[$L^p$-viscosity solutions]\label{def_lpviscosity}
Let $F:S(d)\to\mathbb{R}$ be a fully nonlinear operator satisfying A\ref{A1} and $f\in L^p_{\rm loc}(B_1)$, for $p>d/2$. A function $u \in C(B_1)$ is an $L^{p}$-viscosity sub-solution (resp. super-solution) of 
\[
	F(D^2u)=f\hspace{.2in}\mbox{in}\hspace{.2in}B_1
\]
if, for all $\varphi \in W_{\mathrm{loc}}^{2, p}(B_1)$, whenever $\varepsilon>0$, $U \subset B_1$ is open, and
\begin{align*}
F\left(D^{2} \varphi(x)\right)-f(x) &\geq+\varepsilon \hspace{.2in} \mbox{a.e.}\;x\in U\\
(\mbox{resp. } F\left(D^{2} \varphi(x)\right)-f(x) &\leq-\varepsilon \hspace{.2in} \mbox{a.e.}\;x\in U),
\end{align*}
then $u-\varphi$ cannot have a local maximum (resp. minimum) in $U$. Moreover, if $u$ is both an $L^p$-viscosity sub-solution and an $L^p$-viscosity super-solution, $u$ is said to be an $L^p$-viscosity solution.
\end{definition}

The definition of $L^p$-viscosity solution is necessary since $L^p$-functions might not be defined at the points where the usual conditions must be tested. For a comprehensive account of this notion, we refer the reader to \cite{CafCraKocSwi_1996}. We continue with the definition of weak solutions for double-divergence equations.

\begin{definition}[Weak solution]\label{def_weaksol}
Let $A\in L^\infty(B_1,S(d))$ and denote $A(x)=:[a_{i,j}(x)]_{i,j=1}^d$. Suppose 
\[
	\frac{1}{\lambda} I\leq A(x)\leq \lambda I\hspace{.2in}\mbox{a.e.}-x\in B_1.
\]
We say $m\in L^1(B_1)$ is a weak solution to
\[
	\left(a_{i,j}(x)m\right)_{x_ix_j}=0\hspace{.2in}\mbox{in}\hspace{.2in}B_1
\]
if, for every $\phi\in C^\infty_c(B_1)$ we have
\[
	\int_{B_1}\left(a_{i,j}m\right)\phi_{x_ix_j}\mathrm{d}x=0.
\]
\end{definition}

A solution to the mean-field game in \eqref{eq_mfgmain} combines Definitions \ref{def_lpviscosity} and \ref{def_weaksol}. 

\begin{definition}[Solution for the MFG system]\label{def_solmfg}
The pair $(u,m)$ is a weak solution to \eqref{eq_mfgmain} if the following hold:
\begin{enumerate}
\item We have $u\in C(B_1)\cap W^{1,p}_g$ and $m\in L^1(B_1)$, with $m\geq 0$;
\item The function $u$ is an $L^p$-viscosity solution to 
\[
	F(D^2u)=m^\frac{1}{p-1}\hspace{.2in}\mbox{in}\hspace{.2in}B_1\cap\{u>0\};
\]
\item The function $m$ is a weak solution to 
\[
	\left(F_{ij}(D^2u)m\right)_{x_ix_j}=0\hspace{.2in}\mbox{in}\hspace{.2in}B_1\cap\{u>0\}.
\]
\end{enumerate}
\end{definition}

Next, we recall the Poincar\'e's inequality for functions lacking compact support. In particular, we are interested in $u\in W^{1,p}_g(B_1)$.

\begin{lemma}[Poincar\'e's inequality]\label{poincare}
Let $u \in W_{g}^{1, p}\left(B_{1}\right)$ and $C_{p}>0$ be the Poincar\'e's constant associated with $L^{p}\left(B_{1}\right)$ and the dimension $d$. Then for every $C<C_{p}$, there exists $C_{1}\left(C, C_{p}\right)>0$ and $C_{2} \geq 0$ such that
$$
\int_{B_{1}}|D u|^{p} d x-C \int_{B_{1}}|u|^{p} d x+C_{2} \geq C_{1}\left(\int_{B_{1}}|D u|^{p} d x+\int_{B_{1}}|u|^{p} d x\right).
$$
\end{lemma}
For the detailed proof of this fact, we refer the reader to \cite[Lemma 2.7, p. 22]{Dal_1993}. It follows from $u-g\in W^{1,p}_0(B_1)$ and the usual Poincar\'e's inequality. 

In Section \ref{Sec:Existence}, we deal with the existence of minimizers for $\mathcal{F}_{\Lambda,p}$ in $W^{2,p}_{\rm loc}(B_1)\cap W^{1,p}_g(B_1)$. Our reasoning uses the weak lower-semicontinuity of the functional
\begin{equation*}\label{Def:rF0}
	u\mapsto\mathcal{F}_{0,p}[u]:=\int_{B_1}\left( F(D^2u)\right)^p\mathrm{d}x;
\end{equation*}
this is the content of the following lemma.
\begin{lemma}\label{lemma:semicontinuity}
Let $p>d/2$ and suppose A\ref{A2}, A\ref{A3} and A\ref{A4} hold true. Let $\left(u_{n}\right)_{n \in \mathbb{N}} \subset W^{2,p}_{\rm loc}(B_1)\cap W^{1,p}_g(B_1)$ be such that
$$
D^{2} u_{n} \rightharpoonup D^{2} u_{\infty} \quad \text{in} \quad L^{p}\left(B_{1}, S(d)\right),
$$
Then,
$$
	\int_{B_1}\left( F(D^2u_\infty)\right)^p \mathrm{d}x\leq \liminf _{n \rightarrow \infty} \int_{B_1}\left( F(D^2u_n)\right)^p \mathrm{d}x 
$$
\end{lemma}

For the proof of Lemma \ref{lemma:semicontinuity}, we refer to \cite[Proposition 3]{AndPim_2020}.  In what follows, we detail the proof of Theorem \ref{teo_existence}.

\section{Existence of solutions}\label{Sec:Existence}

In this section, we present the proof of Theorem \ref{teo_existence}; we start by establishing the existence of minimizers for \eqref{eq_mainfunc}.
\begin{proposition}[Existence of minimizers]\label{PropositionExistencemin}
Suppose Assumptions A\ref{A2}, A\ref{A3}, and A\ref{A4} are in force and fix $ p>d/2$, arbitrary. Then there exists $u^*\in W^{2,p}_{\rm loc}(B_1)\cap W^{1,p}_g(B_1)$ such that
\[
	\mathcal{F}_{\lambda,p}[u^*]\le\mathcal{F}_{\lambda,p}[u],
\]
for all $u\in W^{2,p}_{\rm loc}(B_1)\cap W^{1,p}_g(B_1)$.
\end{proposition}
\begin{proof}
Under Assumptions A\ref{A2} and A\ref{A3}, the existence of minimizers follows from the direct method in the calculus of variations. We split the argument into three steps. 

\medskip

\noindent{\bf Step 1 - }We first examine
$$
\gamma:=\inf \left\{\mathcal{F}_{\Lambda,p}[u]: u\in W^{2,p}_{\rm loc}(B_1)\cap W^{1,p}_g(B_1)\right\}.
$$
In view of the Remark \ref{Remark:Assumptions}, $\gamma\ge 0$. Furthermore, since $g\in W^{2,p}(B_1)$,
\begin{align*}
\gamma &\leq \mathcal{F}_{\Lambda,p}[g]\\
&\le\int_{B_1}\left( F(D^2g)\right)^p\mathrm{d}x+\Lambda\abs{B_1}\\
&\le\lambda^p\|D^2g\|^{p}_{L^p(B_1)}+\Lambda\abs{B_1}.
\end{align*}
 Hence, $0\le\gamma\leq C(g,\Lambda)<\infty $. Let $\left(u_{n}\right)_{n \in \mathbb{N}} \subset W^{2,p}_{\rm loc}(B_1)\cap W^{1,p}_g(B_1)$ be a minimizing sequence; there exists $N \in \mathbb{N}$ such that
\[
	\mathcal{F}_{\Lambda,p}[u_n] \leq \gamma+1,
\]
for every $n\geq N$. Therefore, for all $n\geq N$,
\begin{align*}
	\left\|D^{2} u_{n}\right\|_{L^{p}\left(B_{1}\right)}
		&\le\lambda\left(\int_{B_1}\left( F(D^2u_n)\right)^p \mathrm{d}x\right)^{\frac1p}\\
		&\leq \lambda\left( \int_{B_{1}}\left[F\left(D^{2} u_{n}\right)\right]^{p} d x+\Lambda\abs{\{u_n>0\}\cap B_1}\right)^{\frac1p}\\
		&\leq C(\gamma,p).
\end{align*}
In the next step the upper bound for $D^2u_n$ builds upon properties of the functional.

\medskip

\noindent{\bf Step 2 - }As a consequence of the former inequality, we infer that $\left(D^{2} u_{n}\right)_{n \in \mathbb{N}}$ is uniformly bounded in $L^{p}\left(B_{1}\right) $. Since $p>d/2$, the embedding $W^{2,p}(B_1)\hookrightarrow W^{1,p}(B_1)$ is compact. Furthermore, we conclude that $\left(u_{n}\right)_{n \in \mathbb{N}}$ is uniformly bounded in $W^{2,p}_{\rm loc}(B_1)\cap W^{1,p}_g(B_1)$; it follows from Lemma \ref{poincare} combined with general facts \cite{CiaMaz_2016}. Hence, there exists $u_{\infty} \in W^{2,p}_{\rm loc}(B_1)\cap W^{1,p}_g(B_1)$ such that 
 \begin{equation}\label{Eq1:PropositionExistencemin}
u_{n} \rightharpoonup u_{\infty} \text{ in } W^{2,p}_{\rm loc}(B_1)\cap W^{1,p}_g(B_1)
\end{equation}
and
\begin{equation}\label{Eq11:PropositionExistence}
 u_n\to u_{\infty} \text{ strongly in } L^p(B_1) .
 \end{equation}
The result follows at once if we ensure that 
\begin{equation}\label{Eq2:PropositionExistencemin}
\int_{B_1}\left( F(D^2u_{\infty})\right)^p\mathrm{d}x\le\liminf_{n\to\infty}\int_{B_1}\left( F(D^2u_n)\right)^p\mathrm{d}x
\end{equation}
and 
\begin{equation}\label{Eq2*:PropositionExistencemin}
	|\{u_{\infty}>0\}\cap B_1|\leq\liminf_{n\to\infty}|\{u_n>0\}\cap B_1|
\end{equation}
hold.  Notice that Lemma \ref{lemma:semicontinuity} combines the convergence mode in \eqref{Eq1:PropositionExistencemin} to yield \eqref{Eq2:PropositionExistencemin}. In the sequel, we establish \eqref{Eq2*:PropositionExistencemin}.

\medskip

\noindent{\bf Step 3 - }Because of the strong convergence \eqref{Eq11:PropositionExistence}, there exists a subsequence, also denoted with $(u_n)_{n\in\mathbb{N}}$, and a negligible subset $\mathcal{N}\subset B_1$, such that  $u_n(x)\to u_\infty(x)$ for every $x$ in $B_1\setminus\mathcal{N}$. As a consequence, if $u_\infty(x)>0$, there exists $N\in\mathbb{N}$ such that $u_n(x)>0$ for every $n\geq N$. If $u_\infty(x)=0$, then $\chi_{\{u_\infty>0\}}(x)=0$. Therefore,
\begin{equation}\label{Eq3:PropositionExistencemin}
	\chi_{\{u_\infty>0\}}(x)\le\liminf_{n\to\infty}\chi_{\{u_n>0\}}(x)
\end{equation}
for almost every $x\in B_1\setminus\mathcal{N}$. 
Hence,
\begin{equation*}
\begin{split}
\abs{\{u_\infty>0\}\cap B_1}
&=\int_{B_1}\chi_{\{u_\infty>0\}}\mathrm{d}x\\
&\le\liminf_{n\to\infty}\int_{B_1}\chi_{\{u_n>0\}}\mathrm{d}x\\
&=\liminf_{n\to\infty}\abs{\{u_n>0\}\cap B_1},
\end{split}
\end{equation*}
which completes the proof.
\end{proof}

\begin{remark}\normalfont
We notice the minimizing sequence $(u_n)_{n\in\mathbb{N}}$ is uniformly bounded in $W^{2,p}(B_1)$. As a consequence, it is also uniformly bounded in some H\"older space. Therefore, we could have used uniform convergence in \eqref{Eq1:PropositionExistencemin}.
\end{remark}

We close this section with the proof of Theorem \ref{teo_existence}.

\begin{proof}[Proof of Theorem \ref{teo_existence}] We split the proof into four steps.

\medskip

\noindent{\bf Step 1 - } Let $u^*\in W^{2,p}_{\rm loc}(B_1)\cap W^{1,p}_g(B_1)$ be the minimizer for \eqref{eq_mainfunc} whose existence follows from Proposition \ref{PropositionExistencemin}. There exists $\mathcal{N}\subset B_1$ such that $D^2u^*(x)$ is well-defined for every $x\in B_1\setminus\mathcal{N}$, with $|\mathcal{N}|=0$. This fact, combined with A\ref{A3}, implies that $F(D^2u^*(x))\geq 0$ for almost every $x\in B_1$. Therefore, $u^*$ satisfies $F(D^2u^*)\geq 0$ in the $L^p$-viscosity sense; see \cite[Lemma 2.6]{CafCraKocSwi_1996}. 

\medskip

\noindent{\bf Step 2 - }By considering a variation of $u^*$ compactly supported in $B_1\cap\{u>0\}$, we obtain
\begin{equation}\label{eq_msol1}
	\int_{B_1\cap\{u>0\}}\left(F_{ij}(D^2u^*)F(D^2u^*)^{p-1}\right)\varphi_{x_ix_j}\mathrm{d}x=0
\end{equation}
for every $\varphi\in C^\infty_c(B_1\cap\{u>0\})$. Set $F(D^2u^*)=:m^\frac{1}{p-1}$; we infer that $m(x)$ is well-defined and satisfies $m(x)\geq 0 $ for almost every $x\in B_1$. In addition, 
\[
	\begin{split}
		\int_{B_1}m(x)\mathrm{d}x&\leq \int_{B_1}1^{p}{\rm d}x+\int_{B_1}\left[F(D^2u^*)^{(p-1)}\right]^{p/(p-1)}{\rm d}x\\
			&\leq C+C(\lambda,\Lambda)\left\|g\right\|_{W^{2,p}(B_1)};
	\end{split}
\]
that is, $m\in L^1(B_1)$. Finally, we notice the integral in \eqref{eq_msol1} is well-defined and leads to
\[
		\int_{B_1\cap\{u>0\}}\left(F_{ij}(D^2u^*)m\right)\varphi_{x_ix_j}\mathrm{d}x=0,
\]
for every $\varphi\in C^\infty_c(B_1\cap\{u>0\})$. 

\medskip

\noindent{\bf Step 3 - }It remains to check that $u^*$ is an $L^p$-viscosity solution to the first equation in \eqref{eq_mfgmain}. The definition of $m$ implies that $u^*$ satisfies
\[
	F(D^2u^*(x))=m(x)^\frac{1}{p-1}
\]
for almost every $x\in B_1\cap\{u^*>0\}$. As before, an application of \cite[Lemma 2.6]{CafCraKocSwi_1996} ends the proof.

\medskip

\noindent{\bf Step 4 - }We prove that $Du^*\in L^r(B_1)$ for every $d<r<dp/(d-p)$. We start by recalling the Gagliardo-Nirenberg inequality for bounded domains. Being $u^*\in W^{2,p}_{\rm loc}(B_1)\cap W^{1,p}_{g}(B_1)$ a minimizer for \eqref{eq_mainfunc}, there exists $C_1,C_2>0$ such that 
\begin{equation}\label{Eq1:LemmaGNInequality}
	\begin{split}
		\|Du^*\|_{L^r(B_{1/2})}&\le C_1(\Lambda/\lambda,d)\left[\left(1+\|D^2g\|^{\alpha}_{L^p(B_{1})}\right)\|u^*\|^{1-\alpha}_{L^{q_1}(B_{1/2})}\right]\\
			&\quad+C_2\|u^*\|_{L^{q_2}(B_{1/2})},
	\end{split}
\end{equation} 
provided
\begin{equation}\label{Eq2:LemmaGNInequality}
\frac{1}{r}=\frac{1}{d}+\left(\frac{1}{p}-\frac{2}{d}\right)\alpha+\frac{1-\alpha}{q_1}
\end{equation}
for some $1/2<\alpha<1$ and $q_2>0$. We notice the $L^p$-norm of $D^2g$ appears in \eqref{Eq2:LemmaGNInequality} because
\[
	\begin{split}
		\left\|D^2u^*\right\|_{L^p(B_1)}&\leq\frac{1}{\lambda}\left(\int_{B_1}F(D^2u^*)^p{\rm d}x\right)^{1/p}\\
			&\leq\left(\int_{B_1}F(D^2g^*)^p{\rm d}x\right)^{1/p}\\
			&\leq\frac{\Lambda}{\lambda}\left\|D^2g^*\right\|_{L^p(B_1)},
	\end{split}
\]
because of Assumption A\ref{A3} and the fact that $g$ is a competitor for $u^*$.

Given $d\geq 2$, $p>d/2$, and $1<r<\infty$. it is always possible to find $\alpha\in(1/2,1)$ and $q_1>1$ such that \eqref{Eq2:LemmaGNInequality} is satisfied. Because $F(D^2u^*)\geq 0$, we know that for every $q>0$ there exists $C>0$ such that
\[
	\sup_{x\in B_{1/2}}u^*(x)\leq C\left\|u\right\|_{L^p(B_1)}\leq C \left\|g\right\|_{W^{2,p}(B_1)};
\]
see \cite[Theorem 4.8, item (2)]{ccbook}. Hence, \eqref{Eq1:LemmaGNInequality} becomes
\[
	\|Du^*\|_{L^r(B_1)}\leq C(\lambda,d,\Lambda,g)
\]
and a straightforward application of Morrey's Theorem completes the proof.
\end{proof}

\begin{remark}[Improved integrability for $m$]\label{rem_impint}\normalfont
Let $(u,m)$ be a weak solution to the fully nonlinear MFG \eqref{eq_mfgmain}. In case $F$ is strictly convex and $p>2$, we claim that $m\in L^\frac{p}{p-1}(B_1)$. In fact, $m$ is defined almost everywhere in $B_1$ as $m=F(D^2u)^{p-1}$. Under the strict convexity of $F$ and $p>2$, solutions to the Euler-Lagrange equation are minimizers for the functional \eqref{eq_mainfunc}. Hence, A\ref{A3} transmits the integrability of $D^2u\in L^p(B_1)$ to $m$, and the claim follows. Compare with \cite{fabesstroock}; see also \cite{bkrsbook}. Re-writing the exponent above as $1+1/(p-1)$ we quantify the improved integrability of $m$ in face of the $L^1$-regime.
\end{remark}

\begin{remark}[Improved regularity for the value function]\label{rem_vf}\normalfont
The value function is $\alpha$-H\"older-continuous, for every $\alpha\in (0,1)$. Hence, the regularity established in the former argument amounts to an improvement of the usual Krylov-Safonov regularity theory implied by uniform ellipticity.
\end{remark}

\section{Information on the free boundary}\label{section_theorem2}

In the sequel, we examine local properties of the free boundary $\partial\{u>0\}$ and present the proof of Theorem \ref{teo_fb}. The following corollary connects the regularity of minimizers with information on the free boundary. We refer to it when proving the first part of Theorem \ref{teo_fb}.

\begin{corollary}\label{cor_finiteper}
Let $x_0\in B_1$ and $0<r<\mathrm{dist}(x_0,\partial B_1)$. Suppose that $u\in W^{2,p}_{\mathrm{loc}}(B_r(x_0))$ is  non-negative and satisfies the following minimality condition: Given $p>d/2$,
\begin{equation}\label{minimalitycondition}
\mathcal{F}_{\Lambda,p}[u,B_r(x_0)]\le\mathcal{F}_{\Lambda,p}[v,B_r(x_0)],
\end{equation}
for every $v\in W^{2,p}_{\mathrm{loc}}(B_r(x_0))$ such that
\[
\begin{cases}
u\le v\hspace{.2in}&\mbox{in}\hspace{.1in} B_r(x_0)\\
u=v\hspace{.2in}&\mbox{on}\hspace{.1in} \partial B_r(x_0).
\end{cases}
\]
\noindent Assume also that A\ref{A1}-A\ref{A4} holds true. There exists $\varepsilon_0>0$ such that, for every $0<\varepsilon\le\varepsilon_0$ one finds a universal constant $C>0$ for which
\begin{equation}\label{conclusion_cor_finiteper}
	\int_{0}^\varepsilon\mathcal{H}^{d-1}\left(\partial^*\{u>t\}\cap B_r(x_0)\right) \mathrm{d}t\le\varepsilon C.
\end{equation}   
\end{corollary}

\begin{proof}

We split the argument into four steps and begin by proving that, for given $0<\varepsilon\le\varepsilon_0$ (fixed and to be chosen later), one gets
\begin{equation}\label{concluison_step1_cor_finiteper}
\int_{B_{r/2}(x_0)\cap\{0< u\le\varepsilon\}}F(D^2u)^p\mathrm dx+\Lambda\left|\{0<u\le \varepsilon\}\cap B_{r/2}(x_0)\right|\le \varepsilon C,
\end{equation}
for some universal constant $C>0$

\medskip

\noindent{\bf Step 1 - }
We begin by fixing a function $\psi\in\mathcal{C}^{\infty}(\mathbb{R}^d)$ such that
\[
\psi(x):=
\begin{cases}
0\hspace{.2in} &\mbox{if }\hspace{.1in}x\in B_{r/2}(x_0)\\
1 &\mbox{if }\hspace{.1in} x\in\mathbb{R}^d\setminus B_{r}(x_0).
\end{cases}
\]
and
\[
|D\psi|,|D^2\psi|\le C_1,
\]
for some universal constant $C_1>0$. For a fixed $0<\varepsilon_0$, consider the functions
\[
	u_{\varepsilon_0}:=(u-\varepsilon_0)^+
\]
and
\[
	\tilde{u}_{\varepsilon_0}:=\psi u+(1-\psi)u_{\varepsilon_0},
\]
which, by the minimality condition \eqref{minimalitycondition}, give
\begin{equation}\label{Eq1:Cor_finiteper}
\begin{split}
\int_{B_r(x_0)}&F(D^2u)^p\mathrm dx+\Lambda\left|\{u>0\}\cap B_r(x_0)\right|\le\\
&\le\int_{B_r(x_0)}F\left(D^2\tilde{u}_{\varepsilon_0}\right)^p\mathrm dx+\Lambda\left|\{\tilde{u}_{\varepsilon_0}>0\}\cap B_r(x_0)\right|.
\end{split}
\end{equation}

\medskip

\noindent{\bf Step 2 - }Now we calculate $F(D^2\tilde{u}_{\varepsilon})^p$ in $B_r(x_0)$; to do so, first notice that $\tilde{u}_{\varepsilon_0}$ can be writen as
\[
	\tilde{u}_{\varepsilon_0}:=\chi_{\{u>\varepsilon_0\}}\left(u-\varepsilon_0(1-\psi)\right)+\chi_{\{0\le u\le\varepsilon_0\}}\left(\psi u\right).
\]
Therefore,
\[
	D^2\tilde{u}_{\varepsilon_0}:=\chi_{\{u>\varepsilon_0\}} \left(D^2u+\varepsilon D^2\psi\right)+\chi_{\{0\le u\le\varepsilon_0\}}\left(\psi F(D^2u)+uD^2\psi+M\right),
\]
where $M:=Du^TD\psi+DuD\psi^T$. Combining the former equality with Assumption A\ref{A3} one gets
\begin{equation}\label{Eq2:Cor_finiteper}
\begin{split}
    F(D^2\tilde u_{\varepsilon_0})^p=&\chi_{\{u >\varepsilon_0\}}(F(D^2u)+\varepsilon_0\lambda|D^2\psi|)^p\\
    &+\chi_{\{0\le u\le\varepsilon_0\}}(\psi F(D^2u)+|uD^2\psi+M|)^p\\
      \le&\chi_{\{u >\varepsilon_0\}}(F(D^2u)+\varepsilon_0\lambda|D^2\psi|)^p\\
      &+\chi_{\{0\le u\le\varepsilon_0\}}(\psi F(D^2u)+|uD^2\psi|+|M|)^p.
\end{split}
\end{equation}
In the next step we detail an involved chain of inequalities used in the argument.

\medskip

\noindent{\bf Step 3 - }We combine \eqref{Eq2:Cor_finiteper} and \eqref{Eq1:Cor_finiteper}, set
\[
\Gamma_\varepsilon^+:=B_r(x_0)\cap\{u>\varepsilon_0\}\hspace{.1in}\mbox{and }\hspace{.1in}\Gamma_\varepsilon:= B_r(x_0)\cap\{0\le u\le\varepsilon_0\},
\]
and resort to Assumption A\ref{A3} to compute
\begin{align*}
	0\ge&\int_{\Gamma^+_{\varepsilon_0}}F(D^2u)^p-F(D^2\tilde u_\varepsilon)^p+\int_{\Gamma_\varepsilon}F(D^2u)^p-F(D^2\tilde u_\varepsilon)^p+\\
		&+\Lambda\left|\{0<u\le\varepsilon\}\cap B_{r/2}(x_0)\right|\\
	\ge&\int_{\Gamma^+_{\varepsilon_0}}F(D^2u)-(F(D^2u)+\varepsilon_0\lambda|D^2\psi|)^p\mathrm{d} x+\\
		&+\int_{\Gamma_{\varepsilon_0}} F(D^2u)^p-\left(\psi F(D^2u)+(\lambda\varepsilon_0|D^2\psi|+|M|)\right)^p\mathrm{d}x\\
		&+\Lambda\left|\{0\le u\le\varepsilon_0\}\cap B_{r/2}(x_0)\right|\\
	\ge&\int_{\Gamma^+_{\varepsilon_0}}F(D^2u)^p-\left(F(D^2u)^p+\lambda C_1F(D^2u)^{p-1}+\lambda C_1O(\varepsilon_0^2)\right)\mathrm{d} x\\
		&+\int_{\Gamma_{\varepsilon_0}}F(D^2u)-\left(\psi^pF(D^2u)+(\lambda C_1\varepsilon_0+|M|)F(D^2u)^{p-1}\right)\mathrm{d}x\\
		&+\int_{\Gamma_{\varepsilon_0}}O((\lambda C_1\varepsilon_0+|M|)^2)\mathrm{d}x+\Lambda\left|\{u>\varepsilon_0\}\cap B_{r/2}(x_0)\right|\\
	\ge&\int_{\Gamma^+_{\varepsilon_0}}(1-\psi^p)F(D^2u)^p\mathrm{d}x +\Lambda\left|\{0\le u\le\varepsilon_0\}\cap B_{r/2}(x_0)\right|\\
		&-\varepsilon_0\lambda C_1\int_{B_{r/2}(x_0)}F(D^2u)^{p-1}\mathrm{d}x-\lambda C_1 O(\varepsilon_0)-\\
		&-\int_{\Gamma_{\varepsilon_0}}\left(|M|F(D^2u)^{p-1}+O((\varepsilon_0\lambda C_1+|M|)^2)\right)\mathrm{d}x.
\end{align*}
Therefore
\begin{equation}\label{Eq1:Nov23}
	0\ge=\int_{\Gamma^+_{\varepsilon_0}}(1-\psi^p)F(D^2u)^p\mathrm{d}x +\Lambda\left|\{0\le u\le\varepsilon_0\}\cap B_{r/2}(x_0)\right|-(A+B),
\end{equation}
where 
\[
A:=\lambda C_1\left(\varepsilon_0\int_{B_{r/2}(x_0)}F(D^2u)^{p-1}\mathrm{d}x+O(\varepsilon_0)\right),
\]
and
\[
B:=\int_{\Gamma_{\varepsilon_0}}\left(|M|F(D^2u)^{p-1}+O((\varepsilon_0\lambda C_1+|M|)^2)\right)\mathrm{d}x.
\]
The argument in the proof of Theorem \ref{teo_existence} ensures that 
\[
	\int_{B_r(x_0)}F(D^2u)^{p-1}<\infty.
\]
Therefore, there exists a universal constant $C_2>0$ such that
\begin{equation}\label{Eq2:Nov23}
A\le C_2 O(\varepsilon_0).
\end{equation}
Recall that $M=Du^TD\psi+DuD\psi^T$. Because of the bounds imposed on $\psi$ and the estimates available for $Du$, we conclude
\[
	|M|\le2C_1\left\|Du\right\|_{L^\infty(B_{1/2})}.
\]
Hence, by requiring $\varepsilon_0<1$, the H\"older inequality and the Theorem \ref{teo_existence} yield
\begin{align*}
\int_{B_{r/2}(x_0)}\left(\varepsilon_0\lambda C_1+|M|\right)^2\mathrm{d}x\le&\int_{B_{r/2}(x_0)}\left(\varepsilon_0^2\lambda^2 C_1^2+4\varepsilon_0\lambda C_1^2|Du|\right)^2\mathrm{d}x\\
&+4C_1\int_{B_{r/2}(x_0)}|Du|^2\mathrm{d}x\\
\le&\varepsilon_0\left(4\lambda C_1^2\|Du\|_{L^p(B_r(x_0))}\|Du\|_{L^{\frac p{p-1}}(B_r(x_0))}\right)\\
&+4C_1\|Du\|_{L^p(B_r(x_0))}\|Du\|_{L^{\frac p{p-1}}(B_r(x_0))}\\
&+\varepsilon_0\lambda_1^2C_1^2\omega_d\\
\le& C_3\varepsilon_0+C_4,
\end{align*}
where $C_3$ and $C_4$ are positive, universal constants. Hence, there exists a universal constant $C_5>0$ such that
\begin{align*}
B&\le2C_1\int_{\Gamma_{\varepsilon_0}}|Du|F(D^2u)^{p-1}\mathrm{d}x+C_5O(\varepsilon_0)\\
&\le2C_1\|Du\|_{L^p(B_r(x_0))}\left(\int_{B_r(x_0)}F(D^2u)^p\mathrm{d}x\right)^{1-\frac 1p}+C_5O(\varepsilon_0)\\
&\le2C_1\lambda\|Du\|_{L^p(B_r(x_0))}(\|D^2u\|_{W^{2,p}(B_r(x_0))})^{1-\frac 1p}+C_5O(\varepsilon_0).\\
\end{align*}
Thus one finds $C_6>0$ a universal constant for which
\begin{equation}\label{Eq3:Nov23}
B\le C_6\|Du\|_{L^p(B_r(x_0))}+C_5O(\varepsilon_0).
\end{equation} 
By combining \eqref{Eq1:Nov23}, \eqref{Eq2:Nov23} and \eqref{Eq3:Nov23} we have that
\begin{equation}
\begin{split}
0\ge&\int_{B_{r/2}(x_0)\cap\{0\le u\le\varepsilon_0\}}F(D^2u)^p\mathrm{d}x+\Lambda\left|\{0\le u\le\varepsilon_0\}\cap B_{r/2}\right|\\
&+C_6\|Du\|_{L^p(B_r(x_0))}+C_7O(\varepsilon_0).
\end{split}
\end{equation}
Finally, set $\varepsilon_0=O(\|Du\|_{L^p})$. Theorem \ref{teo_existence} ensures that for every $0<\varepsilon<\varepsilon_0$ one obtains
\[
0\ge\int_{\{0\le u\le\varepsilon\}\cap B_{r/2}(x_0)}F(D^2u)^p+\Lambda\left|\{0\le u\le\varepsilon\}\cap B_{r/2}(x_0)\right|-\varepsilon C
\] 
where the constant $C>0$ is now universal.

\medskip

\noindent{\bf Step 4 - }Now, Assumption A\ref{A3} yields
\begin{equation*}\label{Eq2:TheoremFinitePerimeter}
\frac1{\lambda^p}\int_{\{0\le u\le\varepsilon\}\cap B_{r/2}(x_0)}\abs{D^2u}^p\mathrm{d}x\le\int_{\{0\le u\le\varepsilon\}\cap B_{r/2}r(x_0)}\left( F(D^2u)\right)^p\mathrm{d}x.
\end{equation*}

To estimate $\|Du\|_{L^p(\{0<u\le\varepsilon\}\cap B_{r/2}(x_0))}$ we recall the Galiardo-Nirenberg inequality for bounded domains. If $u\in W^{2,p}_{\mathrm{loc}}(B_1)$, there exists universal constants $C_4,C_5>0$ such that
\begin{equation}\label{Eq3:Cor_finiteper}
\|Du\|_{L^p(\Gamma_{\varepsilon})}\le
 C_4\|D^2u\|^{\alpha}_{L^p(\Gamma_{\varepsilon})}\|u\|_{L^{q_1}(\Gamma_{\varepsilon})}^{1-\alpha}
+C_5\|u\|_{L^{q_2}(\Gamma_{\varepsilon})},
\end{equation}
provided 
\[
\frac 1p=\frac 1d+\frac1{q_1}-\frac\theta{d(1-\theta)}
\]
for some $1/2<\alpha<1$ and $q_2>0$, where $\Gamma_\varepsilon:=\{0<u\le\varepsilon\}\cap B_{r/2}(x_0))$. Given $p>d/2$ and $d\ge 2$ it is always possible to find $\alpha\in (1/2,1)$ and $q_1>p$ satisfying \eqref{Eq3:Cor_finiteper}, which implies that there exists $C_6:=C_6(\alpha,p,d,\|u\|_{L^p(\Gamma_\varepsilon)})>0$ such that
\[
C_6\|Du\|_{L^p(\Gamma_\varepsilon)}\le\|D^2u\|_{L^p(\Gamma_\varepsilon)}.
\]
Also,
\begin{align*}\label{Eq4:TheoremFinitePerimeter}
\left(\int_{\Gamma_\varepsilon}\abs{Du}\mathrm{d}x\right)^{p}&
\le\abs{\Gamma_\varepsilon}^{p/p^\prime}\int_{\Gamma_\varepsilon}\abs{Du}^p\mathrm{d}x.
\end{align*}
By combining the former inequalities, we get
\begin{equation*}\label{Eq5:TheoremFinitePerimeter}
\int_{\{0\le u\le\varepsilon\}\cap B_{r/2}(x_0)}\abs{Du}\mathrm{d}x<\varepsilon \tilde C 
,\end{equation*}
where $\tilde C:=\tilde C(n,p,\alpha,\lambda,C_1,\|u\|_{L^p(\Gamma_\varepsilon)})>0$ is an universal constant. A straightforward application of the area formula yields
\[
	\int_{0}^\varepsilon\mathcal{H}^{d-1}\left(\partial^*(\{u>t\})\cap B_{r/2}(x_0)\right) \mathrm{d}t\le\varepsilon C
\]
and finishes the proof.
\end{proof}

\subsection{Proof of Theorem \ref{teo_fb}}

In what follows, we organize the previous results and present the proof of Theorem \ref{teo_fb}. The Sobolev regularity of minimizers and its corollary leads to the finite perimeter of the reduced free boundary. 

\begin{proof}[Proof of Theorem \ref{teo_fb}]
Because of Corollary \ref{cor_finiteper}, there exists a sequence $(\delta_n)_{n\in\mathbb{N}}\subset\mathbb{R}$ of real numbers, with $\delta_n\to0$, satisfying
\[
\mathcal{H}^{d-1}(\partial^*(u>\delta_n))\le C,
\]
for every $n\in\mathbb{N}$. Standard convergence results ensure that
\[
\lim_{n\to\infty}\int_{B_1}\chi_{\{u>\delta_n\}}\mathrm{d}x=\int_{B_1}\chi_{\{u>0\}}\mathrm{d}x.
\]
Finally, the lower semi-continuity of the perimeter implies
\[
\mathcal{H}^{d-1}(\partial^*(\{u>0\}))\le C
\]
and yields the conclusion. 
\end{proof}

\section{Perturbation analysis via $\Gamma$-convergence}\label{Section:Gammaconvergence}

This section specializes the operator $F$ to be the norm and considers small values of the parameter $\Lambda$ in \eqref{eq_mainfunc}. We regard the functional 
\begin{equation}\label{eq_fl}
	\mathcal{G}_{\Lambda,p}[v]:=\int_{B_1\cap\{u>0\}}\|D^2v\|^p\mathrm{d}x+\Lambda|\{u>0\}\cap B_1|
\end{equation}
as a free boundary perturbation of
\begin{equation}\label{eq_f0}
	\mathcal{G}_{0,p}[v]:=\int_{B_1}\|D^2v\|^p\mathrm{d}x.
\end{equation}
Denote with $u_\Lambda$ a minimizer for \eqref{eq_fl} and with $u_0$ the minimizer for \eqref{eq_f0}. We are interested in the behavior of $(u_\Lambda)_{\Lambda>0}$, as $\Lambda\to 0$. In particular, we search for the topologies where the convergence $u_\Lambda\to u_0$ is available. Our starting point is a $\Gamma$-convergence result. Namely, we first prove that $\mathcal{G}_{\Lambda,p}\xrightarrow{\Gamma}\mathcal{G}_{0,p}$ as $\Lambda\to 0$. 

Although interesting on its own merits, the $\Gamma$-convergence problem is motivated by its potential consequences on the regularity theory of minimizers to \eqref{eq_fl}. Indeed, we use properties of $\Gamma$-convergence to prove an approximation result. It states that minimizers are close, in a suitable topology, to a minimizer of the $\Gamma$-limit (see Proposition \ref{AppLemma}). This type of approximation result is central to perturbative methods in regularity theory; see, for instance, \cite{caffarelli89, ccbook}. We believe the $\Gamma$-convergence analysis can be used as an ingredient in the study of improved regularity for Hessian-dependent functionals through approximation methods.
We proceed with some auxiliary lemmas.

\begin{lemma}[Equicoerciveness]\label{Lemma:Equicoersiveness}
Let $p>1$ be fixed and $(\Lambda_n)_{n\in\mathbb{N}}$ be a sequence such that $\Lambda_n\to 0$, as $n\to\infty$. Define the functional $\mathcal{G}_{n,p}:L^p(B_1)\to \mathbb{R}$ as
\begin{equation*}
	\mathcal{G}_{n,p}[v]:=\int_{B_1}\|D^2v\|^p\mathrm{d}x+\Lambda_n|\{v>0\}\cap B_1|
\end{equation*}
if $v\in W^{2,p}(B_1)$, and $\mathcal{G}_{n,p}[v]:=+\infty$ in case $v\in L^p(B_1)\setminus W^{2,p}(B_1)$. Let $(u_m)_{m\in\mathbb{N}}\subset L^p(B_1)$ be such that 
\begin{equation}\label{eq_coerc1}
	\mathcal{G}_{n,p}[u_m]\leq C,
\end{equation}
for every $m\in\mathbb{N}$ and some $C>0$. Then $\|u_m\|_{W^{2,p}(B_1)}\leq C$,
uniformly in $m\in\mathbb{N}$, for some $C>0$. 
\end{lemma}
\begin{proof}
It follows from \eqref{eq_coerc1} that 
\[
	\int_{B_1}\|D^2u_m\|^p\mathrm{d}x\leq \mathcal{G}_{n,p}[u_m]\leq C.
\]
By Lemma \ref{poincare} and standard inequalities available for Sobolev spaces \cite{CiaMaz_2016}, there exists $C>0$ such that
\[
\|u_m\|_{W^{2,p}(B_1)}\le C,
\]
uniformly in $m\in\mathbb{N}$.
\end{proof}

Before continuing, we introduce the functional $\mathcal{G}_{0,p}:L^p(B_1)\to\mathbb{R}$, given by
\begin{equation*}
	\mathcal{G}_{0,p}[v]:=\int_{B_1}\|D^2v\|^p\mathrm{d}x
\end{equation*}
if $v\in W^{2,p}(B_1)$, and $\mathcal{G}_{0,p}[v]:=+\infty$ if $v\in L^p(B_1)\setminus W^{2,p}(B_1)$. The next lemma relates $\mathcal{G}_{n,p}$ and $\mathcal{G}_{0,p}$.

\begin{lemma}\label{Lemma:Prop2Gamma}
Let $p>1$ be fixed and $(\Lambda_n)_{n\in\mathbb{N}}$ be a sequence of real numbers so that $\Lambda_n\to 0$, as $n\to\infty$. For each $u\in L^p(B_1)$ there exists a sequence $(u_n)_{n\in\mathbb{N}}\in L^p(B_1)$ converging strongly to $u$ in $L^p(B_1)$, such that
\begin{equation}\label{eq_gamma2}
\lim_{n\to\infty}\mathcal{G}_{n,p}[u_n]=\mathcal{G}_{0,p}[u].
\end{equation}
\end{lemma}
\begin{proof}
Let $u\in L^p(B_1)$ be given and $u_n:=u$, for every $n\in\mathbb{N}$. If $u\in L^p(B_1)\setminus W^{2,p}(B_1)$, we get
\[
	\mathcal{G}_{n,p}[u_n]=+\infty\hspace{.2in}\mbox{and}\hspace{.2in}\mathcal{G}_{0,p}[u]=+\infty,
\]
and \eqref{eq_gamma2} is immediately satisfied. Conversely, suppose $u\in W^{2,p}(B_1)$. In that case, we have
\begin{align*}
	\lim_{n\to\infty}\mathcal{G}_{n,p}[u_n]&=\int_{B_1} \|D^2u\|^p\mathrm{d}x+\lim_{n\to\infty}\Lambda_n|\{u>0\}\cap B_1|=\mathcal{G}_{0,p}[u].
\end{align*}
\end{proof}

\begin{lemma}\label{Lemma:Prop3Gamma}
Let $p>1$ be fixed and $(\Lambda_n)_{n\in\mathbb{N}}$ be a sequence of real numbers so that $\Lambda_n\to 0$, as $n\to\infty$. Given $(u_n)_{n\in\mathbb{N}}\subset L^p(B_1)$ and $u\in L^p(B_1)$, with $u_n\to u$ strongly in $L^p(B_1)$, we have
\begin{equation}\label{Eq2:LemmaProp2Gamma}
	\mathcal{G}_{0,p}[u]\le\liminf_{n\to\infty}\mathcal{G}_{n,p}[u_n].
\end{equation}
\end{lemma}
\begin{proof}
To deduce \eqref{Eq2:LemmaProp2Gamma} from the strong convergence, suppose first $(u_n)_{n\in\mathbb{N}}\subset L^p(B_1)\setminus W^{2,p}(B_1)$. Then 
\[
	\int_{B_1}\|D^2u_n\|^p\mathrm{d}x=+\infty
\] 
and \eqref{Eq2:LemmaProp2Gamma} follows. Otherwise, suppose $(u_n)_{n\in\mathbb{N}}\subset W^{2,p}(B_1)$. 

Through a subsequence, if necessary, we can suppose the $\liminf$ in \eqref{Eq2:LemmaProp2Gamma} is in fact a limit. If such a limit is not finite, then \eqref{Eq2:LemmaProp2Gamma} trivially holds. Suppose otherwise; if this limit is finite, there exists $C>0$ such that
\[
	\begin{split}
		\mathcal{G}_{n,p}[u_n]\leq C
	\end{split}
\]
for every $n\in\mathbb{N}$, large enough (and therefore for every $n\in\mathbb{N}$). As a consequence, $\|D^2u_n\|_{L^p(B_1)}$ is uniformly bounded; evoking once again standard inequalities for Sobolev functions, one infers the existence of a constant $C>0$ such that
\[
	\|u_n\|_{W^{2,p}_{\rm loc}(B_1)}\leq C.
\]

The weakly lower semi-continuity of the $L^p$-norm yields
\begin{align*}
	\mathcal{G}_{0,p}[u]&=\int_{B_1}\|D^2u\|^p\mathrm{d}x\le\liminf_{n\to\infty}\int_{B_1}\|D^2u_n\|^p\mathrm{d}x\leq\liminf_{n\to\infty}\mathcal{G}_{n,p}[u_n]
\end{align*}
and completes the proof.
\end{proof}

By combining Lemmas \ref{Lemma:Equicoersiveness}, \ref{Lemma:Prop2Gamma}, and \ref{Lemma:Prop3Gamma}, we derive the following theorem.

\begin{theorem}[Gamma Convergence]\label{Theorem:Gammaconvergence}
Let $p>d/2$ be fixed and $(\Lambda_n)_{n\in\mathbb{N}}$ be a sequence of real numbers so that $\Lambda_n\to 0$, as $n\to\infty$. Then $\mathcal{G}_{n,p}\xrightarrow{\Gamma}\mathcal{G}_{0,p}$.
\end{theorem}

In the sequel, we explore a consequence of the $\Gamma$-convergence result. It consists of an approximation result by $C^{1,\alpha}$-regular functions.

\subsection{Regular approximations}\label{subsec_regapprox}

We have proved that minimizers for \eqref{eq_mainfunc} are H\"older-continuous. However, the use of $\Gamma$-convergence allows us to arbitrarily approximate minimizers by $C^{1,\alpha}$-regular functions. This is the content of the following proposition
\begin{proposition}[$C^{1,\alpha}$-approximation]\label{AppLemma}
Let $p>d/2$ be fixed. Given $\delta>0$, there exists $\varepsilon>0$ such that, if $\Lambda<\varepsilon$ and $u\in W^{2,p}(B_1)\cap W^{1,p}_g(B_1)$ be a minimizer for \eqref{eq_mainfunc}, one can find $h\in C^{1,\alpha}_{\rm loc}(B_1)$ satisfying
\begin{equation*}\label{Eq1:AppLemma}
\|u-h\|_{W_g^{1,p}(B_1)}<\delta
\end{equation*}
\end{proposition}
\begin{proof}
We use a contradiction argument. Suppose the statement of the proposition is false. In this case, there exist a real number $\delta_0>0$ and sequences $(u_n)_{n\in\mathbb{N}}$ and $(\Lambda_n)_{n\in\mathbb{N}}$ such that 
\[
	\Lambda_n\to 0
\]
as $n\to\infty$,
\[
	\mathcal{G}_{n,p}[u_n]\leq \mathcal{G}_{n,p}[v]
\]
for every $v\in W^{2,p}_{\rm loc}(B_1)\cap W^{1,p}_g(B_1)$ and every $n\in\mathbb{N}$, but 
\begin{equation}\label{eq_absurdo}
	\|u_n-h\|_{W^{1,p}_g(B_1)}>\delta_0,
\end{equation}
for every $h\in C^{1,\alpha}_{\rm loc}(B_1)$, and every $n\in\mathbb{N}$.

However,
\[
	\|u_n\|_{W^{2,p}(B_1)}\leq C\left(\|g\|_{W^{2,p}(B_1)}+1\right),
\]
for some $C>0$. Hence, there exists $u_\infty\in W^{2,p}(B_1)\cap W^{1,p}_g(B_1)$ such that $u_n$ converges $u_\infty$, weakly in $W^{2,p}(B_1)$ and strongly in $W^{1,p}_g(B_1)$. That is tantamount to say that $u_\infty$ is an accumulation point for the sequence $(u_n)_{n\in\mathbb{N}}$.

Because of Theorem \ref{Theorem:Gammaconvergence}, we conclude that $u_\infty$ is a minimizer for $\mathcal{G}_{0,p}$. Previous results in the literature ensure that $u_\infty\in C^{1,\alpha}_{\rm loc}(B_1)$ \cite{AndPim_2020}. By taking $h:=u_\infty$ in \eqref{eq_absurdo}, we get a contradiction and complete the proof.
\end{proof}

\bigskip

\noindent{\bf Acknowledgements} The authors are grateful to Giovanni Bellettini for his comments on the material in this paper. This work was partially supported by the Centre for Mathematics of the University of Coimbra - UIDB/00324/2020, funded by the Portuguese Government through FCT/MCTES. JC is funded by FAPERJ-Brazil (\# E26/202.075/2020). EP is partly funded by FAPERJ-Brazil (Grant \# E26/200.002/2018), ICTP-Trieste and Instituto Serrapilheira (Grant \# 1811-25904). This study was financed in part by the Coordena\c{c}\~ao de Aperfei\c{c}oamento de Pessoal de N\'ivel Superior - Brazil (CAPES) - Finance Code 001.

\bigskip

	\bibliographystyle{plain}
	\bibliography{bibfile}

\def\polhk#1{\setbox0=\hbox{#1}{\ooalign{\hidewidth\lower1.5ex\hbox{`}\hidewidth\crcr\unhbox0}}}
  \def\polhk#1{\setbox0=\hbox{#1}{\ooalign{\hidewidth
  \lower1.5ex\hbox{`}\hidewidth\crcr\unhbox0}}} \def\cprime{$'$}
\begin{thebibliography}{10}

\bibitem{Achdoubook}
Y.~Achdou, P.~Cardaliaguet, F.~Delarue, A.~Porretta, and F.~Santambrogio.
\newblock {\em Mean field games}, volume 2281 of {\em Lecture Notes in
  Mathematics}.
\newblock Springer, Cham; Centro Internazionale Matematico Estivo (C.I.M.E.),
  Florence, [2020] \copyright 2020.
\newblock Edited by Pierre Cardaliaguet and Alessio Porretta, Fondazione
  CIME/CIME Foundation Subseries.

\bibitem{AltCaf_1981}
H.~W. Alt and L.~A. Caffarelli.
\newblock Existence and regularity for a minimum problem with free boundary.
\newblock {\em J. Reine Angew. Math.}, 325:105--144, 1981.

\bibitem{AndPim_2020}
P.~Andrade and E.~A. Pimentel.
\newblock Stationary fully nonlinear mean-field games.
\newblock {\em J. Anal. Math.}, 145(1):335--356, 2021.

\bibitem{avilesgiga2}
P.~Aviles and Y.~Giga.
\newblock A mathematical problem related to the physical theory of liquid
  crystal configurations.
\newblock In {\em Miniconference on geometry and partial differential
  equations, 2 ({C}anberra, 1986)}, volume~12 of {\em Proc. Centre Math. Anal.
  Austral. Nat. Univ.}, pages 1--16. Austral. Nat. Univ., Canberra, 1987.

\bibitem{avilesgiga1}
P.~Aviles and Y.~Giga.
\newblock The distance function and defect energy.
\newblock {\em Proc. Roy. Soc. Edinburgh Sect. A}, 126(5):923--938, 1996.

\bibitem{kohn3}
J.~Bedrossian and R.~V. Kohn.
\newblock Blister patterns and energy minimization in compressed thin films on
  compliant substrates.
\newblock {\em Comm. Pure Appl. Math.}, 68(3):472--510, 2015.

\bibitem{bensoussan}
A.~Bensoussan, J.~Frehse, and P.~Yam.
\newblock {\em Mean field games and mean field type control theory}.
\newblock Springer Briefs in Mathematics. Springer, New York, 2013.

\bibitem{bkrsbook}
V.~I. Bogachev, N.~V. Krylov, M.~R\"{o}ckner, and S.~V. Shaposhnikov.
\newblock {\em Fokker-{P}lanck-{K}olmogorov equations}, volume 207 of {\em
  Mathematical Surveys and Monographs}.
\newblock American Mathematical Society, Providence, RI, 2015.

\bibitem{caffarelli89}
L.~Caffarelli.
\newblock Interior a priori estimates for solutions of fully nonlinear
  equations.
\newblock {\em Ann. of Math. (2)}, 130(1):189--213, 1989.

\bibitem{CafCraKocSwi_1996}
L.~Caffarelli, M.~G. Crandall, M.~Kocan, and A.~\'{S}wi\polhk ech.
\newblock On viscosity solutions of fully nonlinear equations with measurable
  ingredients.
\newblock {\em Comm. Pure Appl. Math.}, 49(4):365--397, 1996.

\bibitem{CafSalbook}
L.~Caffarelli and S.~Salsa.
\newblock {\em A geometric approach to free boundary problems}, volume~68 of
  {\em Graduate Studies in Mathematics}.
\newblock American Mathematical Society, Providence, RI, 2005.

\bibitem{ccbook}
L.~A. Caffarelli and X.~Cabr\'{e}.
\newblock {\em Fully nonlinear elliptic equations}, volume~43 of {\em American
  Mathematical Society Colloquium Publications}.
\newblock American Mathematical Society, Providence, RI, 1995.

\bibitem{cardaliaguet}
P.~Cardaliaguet.
\newblock Notes on mean-field games.
\newblock 2013.

\bibitem{achang2}
S.-Y.~A. Chang, M.~J. Gursky, and P.~C. Yang.
\newblock Regularity of a fourth order nonlinear {PDE} with critical exponent.
\newblock {\em Amer. J. Math.}, 121(2):215--257, 1999.

\bibitem{achang1}
S.-Y.~A. Chang, L.~Wang, and P.~C. Yang.
\newblock A regularity theory of biharmonic maps.
\newblock {\em Comm. Pure Appl. Math.}, 52(9):1113--1137, 1999.

\bibitem{jakobsen_2021}
I.~Chowdhury, E.~R. Jakobsen, and M.~Krupski.
\newblock On fully nonlinear parabolic mean field games with examples of
  nonlocal and local diffusions, 2021.

\bibitem{CiaMaz_2016}
A.~Cianchi and V.~Maz'ya.
\newblock Sobolev inequalities in arbitrary domains.
\newblock {\em Adv. Math.}, 293:644--696, 2016.

\bibitem{maggi1}
S.~Conti and F.~Maggi.
\newblock Confining thin elastic sheets and folding paper.
\newblock {\em Arch. Ration. Mech. Anal.}, 187(1):1--48, 2008.

\bibitem{maggi2}
S.~Conti, F.~Maggi, and S.~M\"{u}ller.
\newblock Rigorous derivation of {F}\"{o}ppl's theory for clamped elastic
  membranes leads to relaxation.
\newblock {\em SIAM J. Math. Anal.}, 38(2):657--680, 2006.

\bibitem{Dal_1993}
G.~Dal~Maso.
\newblock {\em An introduction to {$\Gamma$}-convergence}, volume~8 of {\em
  Progress in Nonlinear Differential Equations and their Applications}.
\newblock Birkh\"{a}user Boston, Inc., Boston, MA, 1993.

\bibitem{fabesstroock}
E.~B. Fabes and D.~W. Stroock.
\newblock The {$L^p$}-integrability of {G}reen's functions and fundamental
  solutions for elliptic and parabolic equations.
\newblock {\em Duke Math. J.}, 51(4):997--1016, 1984.

\bibitem{pim1}
D.~Gomes, E.~Pimentel, and V.~Voskanyan.
\newblock {\em Regularity theory for mean-field game systems}.
\newblock SpringerBriefs in Mathematics. Springer, [Cham], 2016.

\bibitem{kohnicm}
R.~V. Kohn.
\newblock Energy-driven pattern formation.
\newblock In {\em International {C}ongress of {M}athematicians. {V}ol. {I}},
  pages 359--383. Eur. Math. Soc., Z\"{u}rich, 2007.

\bibitem{kohn2}
R.~V. Kohn and E.~O'Brien.
\newblock The wrinkling of a twisted ribbon.
\newblock {\em J. Nonlinear Sci.}, 28(4):1221--1249, 2018.

\bibitem{ll1}
J.-M. Lasry and P.-L. Lions.
\newblock Jeux \`a champ moyen. {I}. {L}e cas stationnaire.
\newblock {\em C. R. Math. Acad. Sci. Paris}, 343(9):619--625, 2006.

\bibitem{ll2}
J.-M. Lasry and P.-L. Lions.
\newblock Jeux \`a champ moyen. {II}. {H}orizon fini et contr\^{o}le optimal.
\newblock {\em C. R. Math. Acad. Sci. Paris}, 343(10):679--684, 2006.

\bibitem{ll3}
J.-M. Lasry and P.-L. Lions.
\newblock Mean field games.
\newblock {\em Jpn. J. Math.}, 2(1):229--260, 2007.

\bibitem{LCDF}
P.-L. Lions.
\newblock Cours au coll\`ege de france.
\newblock www.college-de-france.fr.

\bibitem{venkataramani}
S.~Venkataramani.
\newblock Lower bounds for the energy in a crumpled elastic sheet---a minimal
  ridge.
\newblock {\em Nonlinearity}, 17(1):301--312, 2004.

\end{thebibliography}

\bigskip

\noindent\textsc{Julio C. Correa-Hoyos}\\
Instituto de Matem\'atica e Estat\'istica\\Universidade do Estado do Rio de Janeiro\\20550-013, Maracan\~a, Rio de Janeiro - RJ, Brazil.\\\noindent\texttt{julio.correa@ime.uerj.br}.
\bigskip

\noindent\textsc{Edgard A. Pimentel (Corresponding Author)}\\
University of Coimbra\\
CMUC, Department of Mathematics\\ 
3001-501 Coimbra, Portugal\\
\noindent\texttt{edgard.pimentel@mat.uc.pt}

\end{document}